\documentclass[12pt]{amsproc}
\usepackage{amsmath,amsthm}
\usepackage{amsfonts,amssymb}
\newtheorem{thm}{Theorem}[section]

\newtheorem{lem}[thm]{Lemma}
\newtheorem{prop}[thm]{Proposition}

\theoremstyle{remark}
\newtheorem{rem}{Remark}[section]

\def\C{{\mathbb C}}
\def\H{{\mathbb H}}
\def\N{{\mathbb N}}
\def\R{{\mathbb R}}
\def\Z{{\mathbb Z}}

\begin{document}

\title[Beurling's theorem]
{Variations on a theorem of Beurling}
\author{Rahul Garg and Sundaram Thangavelu}

\address{Department of Mathematics\\ Indian Institute
of Science\\Bangalore-560 012}
\email{veluma@math.iisc.ernet.in, rahulgarg@math.iisc.ernet.in}

\date{\today}
\thanks{}

\begin{abstract}
We consider functions satisfying the subcritical Beurling's condition, viz., 
$$ \int_{\R^n}\int_{\R^n} |f(x)| |\hat{f}(y)| e^{a |x \cdot y|} \, dx \, dy < \infty $$
for some $ 0 < a < 1.$ We show that such functions are entire vectors for the Schr\"{o}dinger representations of the Heisenberg group. If an eigenfunction $f$ of the Fourier transform satisfies the above condition we show that the Hermite coefficients of $ f $ have certain exponential decay which depends on $a$.
\end{abstract}

\maketitle

\section{Introduction}\label{intro}

For a function $f \in L^1(\R^n)$, define the Fourier transform of $f$ by
$$\hat{f}(y) = (2\pi)^{-n/2} \int_{\R^n} f(x) e^{-ix \cdot y} \, dx.$$
In 1991 L. H\"{o}rmander \cite{H} published a theorem on Fourier transform pairs which he attributed to A. Beurling. This result reads as follows: If an integrable function $ f $ on $\R$ satisfies
$$\int_{\R}\int_{\R} |f(x)| |\hat{f}(y)| e^{|xy|} \, dx \, dy < \infty$$
then it vanishes almost everywhere. Clearly, this result can be viewed as an uncertainty principle for the Fourier transform. As corollaries we can obtain the well known uncertainty principles of Hardy, Cowling-Price and Gelfand-Shilov (see \cite{Th2}). In 2003 Bonami et al \cite{BDJ} extended the above theorem of Beurling to higher dimensions and obtained a far reaching generalisation. They have proved the following result.

\begin{thm}[Bonami-Demange-Jaming] \label{BDJ} A function $f \in L^2(\R^n)$ 
satisfies the condition
$$ \int_{\mathbb{R}^n}\int_{\mathbb{R}^n} \frac{|f(x)| |\hat{f}(y)|}{(1+|x|+|y|)^{N}}
e^{|x \cdot y|} \, dx \, dy < \infty $$
for some $N \geq 0$ if and only if $ f(x) = P(x) e^{-(Ax,x)} $  where $A$ is a real positive definite symmetric matrix and $P$ is a polynomial with $\deg(P) < \frac{N-n}{2} $.
\end{thm}

One motivation for introducing the polynomial factor $(1+|x|+|y|)^{N}$ in the denominator of the integrand in Theorem \ref{BDJ} is to capture the Gaussian $\varphi_0(x) = e^{-\frac{1}{2}|x|^2} $ which does not satisfy the Beurling's condition. However, the same Gaussian also satisfies 
\begin{eqnarray*}
K_a(f) = \int_{\R^n}\int_{\R^n} |f(x)| |\hat{f}(y)| e^{a |x \cdot y|} \, dx \, dy < \infty
\end{eqnarray*}
for any $ 0 < a < 1.$ Note that this amounts to replacing the factor $(1+|x|+|y|)^{-N}$ by $e^{-\delta |x \cdot y|}$ for some $\delta > 0$ in the hypothesis of Theorem \ref{BDJ}. It is therefore natural to consider functions satisfying $ K_a(f) \leq C K_a(Q\varphi_0)$ for all $0<a<1,$ for some polynomial $Q$. It turns out that this condition is actually equivalent to the hypothesis of Theorem \ref{BDJ}.

In view of this it is natural to ask for a characterisation of all functions $ f $ satisfying $ K_a(f) < \infty $ for a fixed $ 0 < a < 1.$ An analogue of this problem in the context of Hardy's theorem has been studied by Demange \cite{D}, Vemuri \cite{V} and the authors \cite{RT1}. Recall the statement of Hardy's theorem \cite{Ha}: If
$$ |f(x)| \leq C e^{-a|x|^2},~~~~ |\hat{f}(y)| \leq C e^{-b|y|^2} $$
then (i) $ f = 0 $ when $ ab > \frac{1}{4} ,$ (ii) $ f(x) = C e^{-a|x|^2 } $ when $ ab = \frac{1}{4} $ and (iii) when $ ab < \frac{1}{4} $ there are infinitely many linearly independent functions (e.g. suitable dilates of  Hermite functions)  satisfying the hypotheses. In \cite{V} Vemuri showed that on $ \R $ the Hermite coefficients of any function $ f $ satisfying Hardy's conditions with $ a = b = \frac{1}{2}\tanh(2t) < \frac{1}{2}$ have the decay $ |(f,h_k)| \leq C (2k+1)^{-1/4}e^{-(2k+1)t/2}$ for all $k \in \N = \{ 0, 1, 2, \ldots \}.$ Here $ h_k $ are the Hermite functions on $ \R.$

It has been conjectured that a similar result is also true in higher dimensions but so far only some partial results have been proved, see \cite{RT1}. In a similar fashion we can ask if the Hermite coefficients of a function satisfying the subcritical Beurling's condition $ K_a(f) < \infty $ will have a certain prescribed exponential decay. It turns out that an exact analogue of Vemuri's theorem fails in this case. One of the aims of this article is to demonstrate this phenomenon. We still do not know if for a function $f$ satisfying Beurling's condition for a fixed $0<a<1,$ there exists some $\epsilon>0$ (depending on both $f$ and $a$) such that $ |(f, \Phi_\alpha)| \leq C_\epsilon e^{-(2|\alpha|+n)\epsilon} $ for all $ \alpha \in \N^n,$ where $\Phi_\alpha$ are the Hermite functions on $\mathbb{R}^n.$ However, for eigenfunctions of the Fourier transform we do obtain exponential decay (depending only on $a$) for the  Hermite coefficients. Also it turns out that any function (not necessarily an eigenfunction of the Fourier transform) satisfying the subcritical Beurling's condition is an entire vector for the Schr\"odinger representation of the Heisenberg group $\H^n$.


\section{Beurling's condition and Hermite coefficients}

As mentioned in the introduction, we start with the following result which we will show later to be equivalent to Theorem \ref{BDJ}.

\begin{thm} \label{*}
Let $f \in L^2(\R^n).$ Then $ K_a(f) \leq C K_a(Q \varphi_0)$
for some polynomial $Q$ and for all $0<a<1$ if and only if $f$ can be written as
$f(x) = P(x) e^{-(Ax , x)}$ for some positive definite matrix $A$ and a polynomial $P$ 
with $\deg(P) \leq \deg(Q)$.
\end{thm}
First we prove the following proposition which is about the bound on the degree of $P$.
%
%
\begin{prop} \label{deg}
Let $R$ and $S$ be polynomials of degree $m_1$ and $m_2$ respectively. If we write
$$E(R,S,a) = \int_{\R^n} \int_{\R^n} |R(x)| |S(y)| e^{-\frac{1}{2}|x|^2} e^{-\frac{1}{2}|y|^2} e^{a |x \cdot y|} \, dx \, dy,$$
then there exist positive constants $C_{R, S}^1$ and $C_{R, S}^2$ such that
$$ C_{R, S}^1 (1-a^2)^{-(n+m_1+m_2)/2} \leq E(R,S,a) \leq C_{R, S}^2 (1-a^2)^{-(n+m_1+m_2)/2}.$$
\end{prop}
\begin{proof}
We first look at the monomials in 1-dimension. Notice that
\begin{eqnarray} \label{monomial}
&& \int_\R \int_\R |x|^j |y|^k e^{-\frac{1}{2}x^2} e^{-\frac{1}{2}y^2} e^{axy} \, dx \, dy \\
\nonumber &=& \int_\R \int_\R |x|^j |y+ax|^k e^{-\frac{1}{2}(1-a^2)x^2} e^{-\frac{1}{2}y^2} \, dx \, dy
\end{eqnarray}
which is bounded above by
\begin{eqnarray*}
&& \sum_{l=0}^k {k \choose l} a^l \int_\R |x|^{j+l} e^{-\frac{1}{2}(1-a^2)x^2} \, dx \int_\R |y|^{k-l}  e^{-\frac{1}{2}y^2} \, dy \\
&=& 4 \sum_{l=0}^k {k \choose l} a^l \int_0^\infty r^{j+l} e^{-\frac{1}{2}(1-a^2)r^2} \, dr \int_0^\infty s^{k-l}  e^{-\frac{1}{2}s^2} \, ds
\end{eqnarray*}
which equals to
$$ 2^{(j+k+2)/2} \sum_{l=0}^{k} \left \{ {k \choose l} \Gamma \left( \frac{j+l+1}{2} \right) \Gamma \left( \frac{k-l+1}{2} \right) \right \} a^{l} (1-a^2)^{-(j+l+1)/2}.$$
The above estimate is the product of $(1-a^2)^{-(j+k+1)/2}$ with
$$ 2^{(j+k+2)/2} \sum_{l=0}^{k} \left \{ {k \choose l} \Gamma \left( \frac{j+l+1}{2} \right) \Gamma \left( \frac{k-l+1}{2} \right) \right \} a^{l} (1-a^2)^{(k-l)/2}.$$
Clearly, the last term is bounded above by a constant which is uniform in $a$ for $0<a<1$. Notice that if we replace $a$ by $-a$ in the integral in (\ref{monomial}), it does not change the value of the integral as the sign is taken care by making a change of variable in either $x$ or $y$-variable. Thus we see that there exists a positive constant, say $c_{j,k}$ such that
\begin{eqnarray} \label{monomials-upper-bound} \int_\R \int_\R |x|^j |y|^k e^{-\frac{1}{2}x^2} e^{-\frac{1}{2}y^2} e^{a|xy|} \, dx \, dy ~\leq~ c_{j,k} (1-a^2)^{-(j+k+1)/2}.
\end{eqnarray}
Let $R(x) = \sum_{|\alpha| \leq m_1} a_\alpha x^\alpha$ and $S(y) = \sum_{|\beta| \leq m_2} b_\beta y^\beta.$ Then $E(R,S,a)$ is bounded above by
\begin{eqnarray*}
&& \sum_{|\alpha| \leq m_1} \sum_{|\beta| \leq m_2} |a_\alpha| \, |b_\beta| \prod_{j=1}^n \int_\R \int_\R |x_j|^{\alpha_j} |y_j|^{\beta_j} e^{-\frac{1}{2}x_j^2} e^{-\frac{1}{2}y_j^2} e^{a|x_j y_j|} \, dx_j \, dy_j \\
&\leq& \sum_{|\alpha| \leq m_1} \sum_{|\beta| \leq m_2} |a_\alpha| \, |b_\beta| \prod_{j=1}^n c_{\alpha_j,\beta_j} (1-a^2)^{-(1+\alpha_j+\beta_j)/2} \\
&\leq& \left(\sum_{|\alpha| \leq m_1} \sum_{|\beta| \leq m_2} |a_\alpha| \, |b_\beta| \prod_{j=1}^n c_{\alpha_j,\beta_j} \right) (1-a^2)^{-(n+m_1+m_2)/2} \\
&=& C_{R, S} (1-a^2)^{-(n+m_1+m_2)/2}
\end{eqnarray*}
where the last inequality is true for all $0<a<1$. This proves our claim on the upper bound of $E(R,S,a)$. Next we prove the claim for the lower bound of $E(R,S,a)$. For this write $R = R_1 + R_2$ where $R_1$ is the homogeneous part of $R$ of leading degree $m_1$. Similarly, write $S = S_1 + S_2$. Clearly, $\deg(R_2) < m_1$ and $\deg(S_2) < m_2$. Now $E(R,S,a)$ is bounded from below by
\begin{eqnarray*}
E(R_1, S_1, a) - \left( E(R_1, S_2, a) + E(R_2, S_1, a) + E(R_2, S_2, a)\right).
\end{eqnarray*}
In view of the upper bound we have already got, there exists a constant $\tilde{C}_{R,S}>0$ such that for all $0<a<1$,
\begin{eqnarray} \label{lower-bound-eq1}
E(R_1, S_2, a) + E(R_2, S_1, a) + E(R_2, S_2, a) \\
\nonumber \leq \tilde{C}_{R,S} (1-a^2)^{-(n+d)/2}
\end{eqnarray}
where $d = \max\{m_1 + \deg(S_2), m_2 + \deg(R_2)\} < m_1 +m_2$. Now we consider $E(R_1, S_1, a)$ which is bounded from below by
$$ \int_{\R^n} \int_{\R^n} |R_1(x)| |S_1(y)| e^{-\frac{1}{2}|x|^2} e^{-\frac{1}{2}|y|^2} e^{a x \cdot y} \, dx \, dy $$
which is equal to
\begin{eqnarray*}
&& \int_{\R^n} \int_{\R^n} |R_1(x)| |S_1(y)| e^{-\frac{1}{2}(1-a^2)|x|^2} e^{-\frac{1}{2}|y-ax|^2} \, dx \, dy \\
&=& a^{-n} \int_{\R^n} \int_{\R^n} |R_1(a^{-1}x)| |S_1(y)| e^{-\frac{1}{2}\left(\frac{1-a^2}{a^2}\right)|x|^2} e^{-\frac{1}{2}|y-x|^2} \, dx \, dy \\
&=& a^{-(n+m_1)} \int_{\R^n} \int_{\R^n} |R_1(x)| |S_1(y+x)| e^{-\frac{1}{2}\left(\frac{1-a^2}{a^2}\right)|x|^2} e^{-\frac{1}{2}|y|^2} \, dx \, dy.
\end{eqnarray*}
Consider the Taylor expansion of $S_1(y+x)$ in $y$-variable about $y=0$,
$$S_1(y+x) = S_1(x) + \sum_{0<|\alpha|\leq m_2} \frac{S_1^{(\alpha)}(x)}{\alpha!} y^\alpha.$$
Therefore, from the above estimate we get
$$E(R_1, S_1, a) \geq a^{-(n+m_1)} \left(I_1(R_1, S_1, a) - I_2(R_1, S_1, a) \right)$$
where
$$ I_1(R_1, S_1, a) = \int_{\R^n} \int_{\R^n} |R_1(x)| |S_1(x)| e^{-\frac{1}{2} \left(\frac{1-a^2}{a^2}\right)|x|^2} e^{-\frac{1}{2}|y|^2} \, dx \, dy $$
while $I_2(R_1, S_1, a)$ is equal to
$$\sum_{0<|\alpha|\leq m_2} \frac{1}{\alpha!} \int_{\R^n} \int_{\R^n} \left|R_1(x)\right| \left|S_1^{(\alpha)}(x)\right| \left|y^\alpha\right| e^{-\frac{1}{2} \left(\frac{1-a^2}{a^2}\right)|x|^2} e^{-\frac{1}{2}|y|^2} \, dx \, dy.$$
Making a change of variable in $x$ and using the fact that $R_1$ and $S_1$ are homogeneous of degree $m_1$ and $m_2$ respectively, we get
\begin{eqnarray*}
I_1(R_1, S_1, a) &=& C'_{R_1,S_1} a^{n+m_1+m_2} (1-a^2)^{-(n+m_1+m_2)/2} ~, \\
I_2(R_1, S_1, a) &=& \left( \sum_{0<|\alpha| \leq m_2}  C''_{R_1, S_1, \alpha} (1-a^2)^{|\alpha|/2} a^{-|\alpha|} \right) \\
&& \times a^{n+m_1+m_2} (1-a^2)^{-(n+m_1+m_2)/2}
\end{eqnarray*}
where
\begin{eqnarray*}
C'_{R_1,S_1} &=& \int_{\R^n} \int_{\R^n} |R_1(x)| |S_1(x)| e^{-\frac{1}{2}|x|^2} e^{-\frac{1}{2}|y|^2} \, dx \, dy, \\
C''_{R_1, S_1, \alpha} &=& \frac{1}{\alpha!} \int_{\R^n} \int_{\R^n} |R_1(x)| |S_1^{(\alpha)}(x)| |y^\alpha| e^{-\frac{1}{2}|x|^2} e^{-\frac{1}{2}|y|^2} \, dx \, dy.
\end{eqnarray*}
Choose $\delta > 0$ (small enough) so that
$$ \sum_{0<|\alpha| \leq m_2} C''_{R_1, S_1, \alpha} (1-a^2)^{|\alpha|/2} a^{-|\alpha|} < C'_{R_1,S_1}$$
for all $1-\delta < a <1$. As a consequence of this we get a positive constant, say $C_{R_1,S_1,\delta}$ such that
$$ E(R_1, S_1, a) \geq C_{R_1,S_1,\delta} (1-a^2)^{-(n+m_1+m_2)/2} $$
for all $1-\delta < a <1$. The above estimate together with (\ref{lower-bound-eq1}) and the fact that $d< m_1+m_2$ implies that there exists some $C_{R, S, \delta}>0$ such that
$$ E(R, S, a) \geq C_{R, S, \delta} (1-a^2)^{-(n+m_1+m_2)/2} $$
for all $1-\delta < a <1$ (for a reduced $\delta>0$, if necessary). For $0 < a \leq 1-\delta$, the above inequality holds easily because in this range of $a$, 
$$ E(R, S, a) \geq \left( E(R,S,0) \left(1-(1-\delta)^2 \right)^{(n+m_1+m_2)/2} \right) (1-a^2)^{-(n+m_1+m_2)/2}.$$
This completes the claim for lower bound on $E(R, S, a)$.
\end{proof}
%
%
A similar proposition about the bound on the degree of the polynomial $P$ in Theorem \ref{BDJ} was proved in \cite{BDJ} (Proposition 2.1, \cite{BDJ}). For the convenience of the 
reader, we state it here:
\begin{prop} \label{prop-degree-BDJ}
Let $f(x) = P(x) e^{-\frac{1}{2}((A+iB)x, x)}$ with $A$ and $B$ two real symmetric matrices and $P$ a polynomial. If $ f \in L^2(\R^n),$ then $A$ is positive definite. 
Moreover, the following are equivalent:
\begin{flalign*}
&(i)~~ \int_{\R^n} \int_{\R^n} \frac{|f(\xi)| |\hat{f}(\eta)|}{(1+ |\xi| + |\eta|)^N} e^{|\xi \cdot \eta|} \, d\xi \, d\eta < \infty~ ;& \\
&(ii)~~ B=0~~ \mbox{and} ~~ \deg(P) < \frac{N-n}{2}.&
\end{flalign*}
\end{prop}

We are now in a position to prove Theorem \ref{*}.\\

{\bf Proof of Theorem \ref{*}}.
Let $f(x) = P(x) e^{-(Ax, x)}$ for some positive definite matrix $A$ and a polynomial $P$ with $\deg(P) = k$. By definition of $K_a(f),$ it immediately follows that $K_a(f) = K_a(f_M)$ where $f_M(x) = |\det M|^{1/2} f(Mx),$ $M$ being an invertible matrix. Therefore, without loss of generality, we can assume that $A = \frac{1}{2}I.$ In that case $\hat{f}$ takes the form $\hat{f}(y) = H(y) e^{-\frac{1}{2} |y|^2}$ for some polynomial $H$ with $\deg(H) = \deg(P) =k$.
Thus, by Proposition \ref{deg}, there exists a positive constant, say $C_1$ such that
$$ K_a(f) = E(P,H,a) \leq C_1 (1-a^2)^{-(n+2k)/2}.$$
Similarly for a given polynomial $Q$ with $\deg(Q) = m$, in view of the lower bound estimate of Proposition \ref{deg} there exists a positive constant, say $C_2$ such that
$$ K_a(Q \varphi_0) \geq C_2 (1-a^2)^{-(n+2m)/2}.$$
If $k \leq m$, then $(1-a^2)^{m-k} \leq 1$ for all $0<a<1$ and therefore,
\begin{eqnarray*}
K_a(f) &\leq& C_1 (1-a^2)^{-(n+2k)/2}\\
&\leq& C_1 (1-a^2)^{-(n+2m)/2} \leq C_1 C_2^{-1} K_a(Q \varphi_0).
\end{eqnarray*}
Conversely, suppose that $K_a(f) \leq C K_a(Q \varphi_0)$ for some polynomial $Q$ with $\deg(Q) = k$. As above, this implies that for all $0<a<1$
$$K_a(f) \leq \tilde{C} (1-a^2)^{-(n+2k)/2} \leq \tilde{C} (1-a)^{-(n+2k)/2}$$
for some positive constant $\tilde{C}$. We will prove that $f(x) = P(x) e^{-(Ax , x)}$ for some positive definite matrix $A$ and a polynomial $P$. Then the fact that $\deg(P) \leq k$ will follow from Proposition \ref{deg}. The proof of this theorem is adapted from that of Demange in \cite{D}. We make use of the Bargmann transform $ B $ which takes $ L^2(\R^n)$ isometrically onto the Fock space consisting of all entire functions on $ \C^n $ that are square integrable with respect to the Gaussian measure $ (4\pi)^{-n/2}
e^{-\frac{1}{2} |z|^2} dz. $ The transform $ B $ is explicitly given by
$$ Bg(z)= \pi^{-n/2}e^{- \frac {1}{4}z^2} \int_{\mathbb{R}^n}g(\xi)~ e^{- \frac{1}{2} |\xi|^2} e^{z \cdot \xi} \, d\xi $$
where $g \in L^2(\R^n)$ and $z \in \C^n .$ Throughout this article by $z \cdot \xi$ we mean $x \cdot \xi + i y \cdot \xi.$ In estimating $ Bf $ we make use of another important property of the Bargmann transform, viz., $ Bg(-iz) = B\hat{g}(z).$ We apply the Bargmann transform on $f$ and $\hat{f}$ simultaneously to get
$$ Bf(z)B\hat{f}(z) =  \pi^{-n}  e^{- \frac {1}{2}z^2} 
\int_{\R^n} \int_{\R^n} f(\xi) \hat{f}(\eta)~ 
e^{- \frac{1}{2}(|\xi|^2 + |\eta|^2)} e^{z \cdot (\xi + \eta)} \, d\xi \, d\eta $$
which is bounded by
$$ \pi^{-n}  e^{-\frac{1}{2}(|x|^2-|y|^2)} \int_{\R^n} \int_{\R^n} |f(\xi)| |\hat{f}(\eta)|~ e^{- \frac{1}{2}(|\xi|^2 + |\eta|^2)} e^{x \cdot (\xi + \eta)} \, d\xi \, d\eta.$$
Now $e^{- \frac{1}{2}(|\xi|^2 + |\eta|^2)} 
e^{x \cdot (\xi + \eta)} $ can be written as
$$  e^{a \xi \cdot \eta} e^{-\frac{1}{2}|\xi-(x-a\eta)|^2} e^{-\frac{1}{2}(1-a^2) \left |\eta-\frac{1}{1+a}x \right|^2} e^{\frac{1}{2} \left( \frac{1-a}{1+a} \right)|x|^2} e^{\frac{1}{2}|x|^2} $$
which is bounded by $ e^{a |\xi \cdot \eta|} e^{\frac{1}{2} \left( \frac{1-a}{1+a} \right)|x|^2} e^{\frac{1}{2}|x|^2} .$ Therefore,
\begin{eqnarray}
\nonumber |Bf(z)B\hat{f}(z)| &\leq& \pi^{-n} K_a(f) \exp \left( \frac{1}{2} \left(|y|^2 + \frac{1-a}{1+a}|x|^2 \right) \right)\\
&\leq& C_1 (1-a)^{-(n+2k)/2} \exp \left( \frac{1}{2} \left(|y|^2 + \frac{1-a}{1+a}|x|^2 \right) \right). \label{BDJ-type-eq1}
\end{eqnarray}
The above estimate (\ref{BDJ-type-eq1}) holds for all $z \in \C^n $ and for all $0<a<1.$ Now for $z = x+iy$ with $|x|>1,$ we choose $a$ such that $(1-a)|x|^2 = 1.$ Then (\ref{BDJ-type-eq1}) becomes
\begin{eqnarray}
\nonumber |Bf(z)B\hat{f}(z)| &\leq& C_1 |x|^{n+2k}\exp \left( \frac{1}{2} \left(|y|^2 + \frac{1}{1+a} \right) \right)\\
&\leq& C_1 e^{1/2} (1+|x|)^{n+2k} \exp \left( \frac{1}{2} |y|^2\right). \label{BDJ-type-eq2}
\end{eqnarray}
Now for all $|x| \leq 1,$ (\ref{BDJ-type-eq1}) gives
\begin{eqnarray*}
|Bf(z)B\hat{f}(z)| &\leq& C_1 (1-a)^{-(n+2k)/2} \exp \left( \frac{1}{2} \left(|y|^2 + \frac{1-a}{1+a} \right) \right) \\
&\leq& C_1 e^{1/2} (1-a)^{-(n+2k)/2} \exp \left( \frac{1}{2} |y|^2 \right)
\end{eqnarray*}
which holds for all $a$ with $0<a<1.$ So we can let $a \to 0$ in the above estimate and then together with (\ref{BDJ-type-eq2}), we get for all $z \in \C^n,$
\begin{eqnarray} \label{BDJ-type-eq3}
|Bf(z)B\hat{f}(z)| \leq C_2 (1+|x|)^{n+2k} \exp \left( \frac{1}{2} |y|^2 \right)
\end{eqnarray}
where $C_2 = C_1 e^{1/2}.$ Similar arguments can be given to verify that $Bf(z)B\hat{f}(-z)$ also satisfies the estimate of (\ref{BDJ-type-eq3}). Now Since $K_a(f) = K_a(\hat{f}),$ the same estimate (\ref{BDJ-type-eq3}) holds when $f$ is replaced by $\hat{f}$ and then in view of $B\hat{f}(z) = Bf(-iz),$ we have
\begin{eqnarray} \label{BDJ-type-eq4}
|Bf(z)B\hat{f}(z)| \leq C_2 (1+|y|)^{n+2k} \exp \left( \frac{1}{2} |x|^2\right).
\end{eqnarray}
Now we recall the following variant of Phragm$\acute{\textup{e}}$n-Lindel\"{o}f principle:
\begin{lem} \label{Phrag-Lind} Suppose $G \colon \C^n \to C$ be an entire function which satisfies the following estimates:
\begin{eqnarray*}
|G(z)| &\leq& C (1+|z|)^m e^{b|y|^2},\\
|G(\xi)| &\leq& C (1+|\xi|)^m e^{-b|\xi|^2}
\end{eqnarray*}
for all $z=x+iy \in \C^n, \xi \in \R^n$. Then $G(z) = P(z) e^{-bz^2}$ where $P$ is a polynomial of degree at most $m$.
\end{lem}

For a proof of this result we refer to Theorems 1.4.4, 1.4.5 on Page 22-23 in \cite{Th2}. Let us write $F(z) = e^{-\frac{1}{2}z^2} Bf(z)B\hat{f}(z)$. Then for all $\xi \in \R^n$
\begin{eqnarray*}
|F(\xi)| &=& e^{-\frac{1}{2}|\xi|^2} |Bf(\xi)B\hat{f}(\xi)| \\
&\leq& C_2 (1+|\xi|)^{n+2k} e^{-\frac{1}{2}|\xi|^2}
\end{eqnarray*}
using the estimate (\ref{BDJ-type-eq3}). And for all $z \in \C^n$
\begin{eqnarray*}
|F(z)| &=& e^{-\frac{1}{2}(|x|^2-|y|^2)} |Bf(z)B\hat{f}(z)| \\
&\leq& C_2 (1+|y|)^{n+2k} e^{\frac{1}{2}|y|^2}
\end{eqnarray*}
using the estimate (\ref{BDJ-type-eq4}). Therefore, Lemma \ref{Phrag-Lind} can be applied to $F$ to conclude that $F(z) e^{\frac{1}{2}z^2}$ is a polynomial of degree at most $n+2k$, which is same as saying that $Bf \cdot B\hat{f}$ is a polynomial of degree at most $n+2k$. Now we recall the following result (Lemma 2.3 in \cite{BDJ})
\begin{lem} \label{BDJ-factorization} Let $\varphi$ be an entire function of order 2 on $\C^n$ such that, on every complex line, either $\varphi$ is identically $0$ or it has at most $N$ zeros. Then, there exists a polynomial $R$ with degree at most $N$ and a polynomial $S$ with degree at most $2$ such that $\varphi(z) = R(z) e^{S(z)}.$
\end{lem}
Since $Bf \cdot B\hat{f}$ is a polynomial, it is easy to see that $Bf$ satisfies the hypothesis of Lemma \ref{BDJ-factorization} and therefore, $Bf(z) = R(z) e^{S(z)}$ where $R$ is a polynomial of degree at most $\frac{n+2k}{2}$ and $S$ is a polynomial of degree at most 2. Now, $Bf(z)B\hat{f}(z) = Bf(z)Bf(-iz) = R(z) R(-iz) e^{S(z)+S(-iz)}.$
But since $Bf \cdot B\hat{f}$ is a polynomial, we must have $S(z)+S(-iz) = 2m(z)\pi i,$ for some $m(z) \in \Z.$ The continuity of $S$ implies that $m \colon \C^n \to \Z$ is continuous. Therefore, $m$ has to be a constant function. Without loss of generality, we can assume that $m=0$, which means $S(z)+S(-iz) = 0.$ This together with the fact that $S$ is a polynomial of degree at most 2 forces $S$ to be a homogeneous polynomial. If we denote the Gaussian function by $\varphi_0(x) = e^{-\frac{1}{2}|x|^2}$, then by definition of the Bargmann transform we have
\begin{eqnarray*}
\hat{f} * \varphi_0 (x) &=& \pi^{n/2} e^{-\frac{1}{4}|x|^2} B\hat{f}(x) \\
&=& \pi^{n/2} R(-ix) e^{-\frac{1}{4}|x|^2 + S(-ix)} \\
&=& \pi^{n/2} R(-ix) e^{-((C+iD)x, x)}.
\end{eqnarray*}
Since $S$ is homogeneous, it follows that $C$ and $D$ are symmetric. Also, integrability of $\hat{f}$ implies integrability of $\hat{f} * \varphi_0$ which forces $C$ to be positive definite. Taking the Fourier transform on both sides we get
$$f(-\xi) \varphi_0(\xi) = \tilde{R}(\xi) e^{-\frac{1}{4}((C+iD)^{-1}\xi, \xi)}$$
for some polynomial $\tilde{R}$ with $\deg(\tilde{R}) = \deg(R)$. Rewriting the above expression, we have
$$f(\xi) = P(\xi) e^{-\frac{1}{2}((A+iB)\xi, \xi)}$$
where $A$ and $B$ are real symmetric matrices. Also, the integrability of $f$ implies that $A$ is positive definite. As described in the beginning of the proof of this theorem that without loss of generality we can assume that $A = I$. Now $\hat{f}(y) = \tilde{Q}(y) e^{-\frac{1}{2} ((I+iB)^{-1}y , y)}$ for some polynomial $\tilde{Q}.$ Notice that $(I+iB)^{-1} = (I-iB)(I+B^2)^{-1}$ and thus $K_a(f) < \infty$ implies that the homogeneous polynomial $(\xi, \xi) + ((I+B^2)^{-1}\eta , \eta) - 2a(\xi, \eta)$ is non-negative. Since this is true for all $0<a<1,$ it follows that the homogeneous polynomial $(\xi, \xi) + ((I+B^2)^{-1}\eta , \eta) - 2(\xi, \eta)$ is non-negative. Our claim is that it is possible if and only if $B=0$. To see this, let $\eta_0$ be an eigenfunction corresponding to an eigenvalue $\lambda_0$ of the real symmetric matrix $B$. Then, $(I + B^2)^{-1} \eta_0 = (1+\lambda_0^2)^{-1} \eta_0$. Thus for all $\xi \in \R^n$,
$$ (\xi, \xi) + (1+\lambda_0^2)^{-1} (\eta_0 , \eta_0) - 2(\xi, \eta_0) \geq 0.$$
In particular, take $\xi = \eta_0$ to get $\left( (1+\lambda_0^2)^{-1} -1\right) \|\eta_0\|^2 \geq 0,$ which implies $\lambda_0 = 0$. This means the only eigenvalue of the symmetric matrix $B$ is 0 which proves that $B=0$. This completes the proof of the theorem.

We will now show that the condition in Theorem \ref{*} and that in Theorem \ref{BDJ} are equivalent. This follows from Propositions \ref{deg} and \ref{prop-degree-BDJ} and the following two lemmas (Lemmas \ref{***} and \ref{****}).

\begin{lem} \label{***}
Let $f \in L^2(\R^n)$ be such that for some $N \geq 0$
$$ E = \int_{\R^n}\int_{\R^n} \frac{|f(\xi)| |\hat{f}(\eta)|}{(1+|\xi|+|\eta|)^{N}}
e^{|\xi \cdot \eta|} \, d\xi \, d\eta < \infty .$$
Then there exists a positive constant $C$ such that for all $0<a<1$
$$ K_a(f) \leq C(1-a^2)^{-N} $$
\end{lem}
\begin{proof}
Without loss of generality assume that $f \neq 0.$
Write
$$A(\eta) = \int_{\R^n} \frac{|f(\xi)|} {(1+|\xi|)^{N}} e^{|\xi \cdot \eta|} \, d\xi.$$
Then,
$$ \int_{\R^n} \frac{A(\eta)}{(1+|\eta|)^{N}} |\hat{f}(\eta)| \, d\eta \leq E < \infty $$
and thus $A(\eta) \hat{f}(\eta)$ is finite almost everywhere. Also since $\hat{f} (\neq 0)$ is finite almost everywhere (being in $L^2(\R^n)),$ it follows that $|U| > 0$ where $U = \{ \eta \in \R^n : A(\eta) < \infty\}$ and $|U|$ is the Lebesgue measure of $U.$ Now we claim that there exists $n$ linearly independent vectors in $U.$ For this choose $\eta^1 \in U$. Now $|U \setminus \R \eta^1| = |U| >0.$ Now choose $\eta^2 \in U \setminus \R \eta^1.$ Again $|U \setminus (\R \eta^1 + \R \eta^2)| = |U| >0.$ Continuing this way we can get a linearly independent set, say $\{ \eta^1, \eta^2, ...., \eta^n \}$ in $U.$ Let $A = (a_{ij})$ be the $n \times n$ matrix whose $j^{th}$ row forms the vector $\eta^j$ and let $B = (b_{ij})$ be the inverse of $A.$ Let $m' = \max_{1 \leq i,j \leq n} |b_{ij}|$, then clearly $m' >0.$
Also by definition of $U,$ $A(\eta^j) < \infty$ which in particular implies
$$ \int_{\R^n} \frac{|f(\xi)|}{(1+|\xi|)^N} e^{|a_{j1} \xi_1 +....+ a_{jn} \xi_n|} \, d\xi < \infty $$
for each $j = 1,2,....,n.$ Now,
$$ \left| \xi_k \right| = \left| \sum_{j=1}^n b_{kj} (a_{j1} \xi_1 +....+ a_{jn} \xi_n) \right| \leq m' \sum_{j=1}^n  |a_{j1} \xi_1 +....+ a_{jn} \xi_n|.$$
This implies
\begin{eqnarray*}
e^{\frac{1}{m' n}|\xi_k|} \leq \prod_{j=1}^n \left( e^{|a_{j1} \xi_1 +....+ a_{jn} \xi_n|} \right)^{1/n} \leq \frac{1}{n} \sum_{j=1}^n e^{|a_{j1} \xi_1 +....+ a_{jn} \xi_n|}
\end{eqnarray*}
and therefore, for each $k = 1,2,....,n,$
$$ \int_{\R^n} \frac{|f(\xi)|} {(1+|\xi|)^{N}} e^{\frac{1}{m' n}|\xi_k|} \, d\xi < \infty. $$
Again,
$$ e^{\frac{1}{m' n^2}|\xi|} \leq \prod_{j=1}^n \left( e^{\frac{1}{m' n}|\xi_j|} \right)^{1/n} \leq \frac{1}{n} \sum_{j=1}^n e^{\frac{1}{m' n}|\xi_j|} $$
which gives
$$ \int_{\R^n} \frac{|f(\xi)|} {(1+|\xi|)^{N}} e^{\frac{1}{m' n^2}|\xi|} \, d\xi < \infty. $$
Similarly,
$$ \int_{\R^n} \frac{|\hat{f}(\eta)|} {(1+|\eta|)^{N}} e^{\frac{1}{m'' n^2}|\eta|} \, d\eta < \infty $$
for some $m'' > 0.$ Let us fix $0 < \delta < \min\{\frac{1}{m' n^2}, \frac{1}{m'' n^2} \}$ so that we have
\begin{eqnarray*}
R_1 &=& \int_{\R^n} |f(\xi)| e^{\delta|\xi|} \, d\xi < \infty,\\
R_2 &=& \int_{\R^n} |\hat{f}(\eta)| e^{\delta|\eta|} \, d\eta < \infty .
\end{eqnarray*}
Let $U_{a,\delta} = \{(\xi,\eta) \in \R^n \times \R^n : |\xi \cdot \eta| \geq \frac{\delta}{a}(1+|\xi|+|\eta|)\},~ V_{a,\delta} = \{(\xi,\eta) \in \R^n \times \R^n : |\xi \cdot \eta| < \frac{\delta}{a}(1+|\xi|+|\eta|)\}$ and write
\begin{eqnarray*}
K_a^1(f) &=& \int \int_{U_{a,\delta}} |f(\xi)| |\hat{f}(\eta)| e^{a|\xi \cdot \eta|} \, d\xi \, d\eta,\\
K_a^2(f) &=& \int \int_{V_{a,\delta}} |f(\xi)| |\hat{f}(\eta)| e^{a|\xi \cdot \eta|} \, d\xi \, d\eta.
\end{eqnarray*}
Then $K_a(f) = K_a^1(f) + K_a^2(f)$ and
$$ K_a^2(f) \leq \int_{\R^n} \int_{\R^n} |f(\xi)| |\hat{f}(\eta)| e^{\delta(1+|\xi|+|\eta|)} \, d\xi \, d\eta = R_1 R_2 e^{\delta}.$$
Now,
$$ K_a^1(f) \leq \int \int_{U_{a,\delta}} \frac{|f(\xi)| |\hat{f}(\eta)|}{(1+|\xi|+|\eta|)^N} \left( \frac{a}{\delta}|\xi \cdot \eta| \right)^N e^{a|\xi \cdot \eta|} \, d\xi \, d\eta $$
which is bounded by the product of $\left( \frac{a}{\delta(1-a)} \right)^N$ and
$$ \int_{\R^n} \int_{\R^n} \frac{|f(\xi)| |\hat{f}(\eta)|}{(1+|\xi|+|\eta|)^N} e^{|\xi \cdot \eta|} \left \{ ((1-a)|\xi \cdot \eta|)^N e^{-(1-a)|\xi \cdot \eta|}\right \} \, d\xi \, d\eta.$$
But we know that for each fixed $N \geq 0,$ $\sup_{r \geq 0}r^N e^{-r} = M < \infty.$ Thus
\begin{eqnarray*}
K_a^1(f) &\leq& \frac{a^N M}{\delta^{N}} (1-a)^{-N} \int_{\R^n} \int_{\R^n} \frac{|f(\xi)| |\hat{f}(\eta)|}{(1+|\xi|+|\eta|)^{N}}
e^{|\xi \cdot \eta|} \, d\xi \, d\eta \\
&=& \frac{a^N M E}{\delta^{N}} (1-a)^{-N} \\
&\leq& M' (1-a^2)^{-N} \hspace{2cm} (\textup{since }0<a<1).
\end{eqnarray*}
If we choose $C = M' + R_1 R_2 e^{\delta}$, then we have
$$ K_a(f) \leq C(1-a^2)^{-N} $$
for all $0<a<1$.
\end{proof}

\begin{lem} \label{****}
Let $f \in L^2(\R^n)$ be such that for all $0<a<1$
$$ K_a(f) \leq C(1-a^2)^{-M} $$
for some $C>0,~ M \geq 0$. Then
$$ E = \int_{\R^n}\int_{\R^n} \frac{|f(\xi)| |\hat{f}(\eta)|}{(1+|\xi|+|\eta|)^{2M+3}} e^{|\xi \cdot \eta|} \, d\xi \, d\eta < \infty.$$
\end{lem}
\begin{proof}
It follows from the definition of $K_a(f)$ that $f,~ \hat{f} \in L^1(\R^n).$ Let $U = \{(\xi,\eta) \in \R^n \times \R^n : |\xi \cdot \eta| \leq 1\}.$ For each $k \in \{2,3,4,....... \}$ write $U_k = \{(\xi,\eta) \in \R^n \times \R^n : \frac{k}{2} \leq |\xi \cdot \eta| \leq k\}$ and
$$E_k = \int \int_{U_k} \frac{|f(\xi)| |\hat{f}(\eta)|}{(1+|\xi|+|\eta|)^{2M+3}}
e^{|\xi \cdot \eta|} \, d\xi \, d\eta. $$
Now since $|\xi|+|\eta| \geq (2|\xi| |\eta|)^{1/2} \geq (2|\xi \cdot \eta|)^{1/2},$ we have
\begin{eqnarray*}
E_k &=& \int \int_{U_k} \frac{|f(\xi)| |\hat{f}(\eta)|}{(1+|\xi|+|\eta|)^{2M+3}}
e^{(1-\frac{1}{k})|\xi \cdot \eta|} e^{\frac{1}{k}|\xi \cdot \eta|} \, d\xi \, d\eta \\
&\leq& e k^{-(2M+3)/2} \int_{\R^n} \int_{\R^n} |f(\xi)| |\hat{f}(\eta)| e^{(1-\frac{1}{k})|\xi \cdot \eta|} \, d\xi \, d\eta \\
&\leq& Ce k^{-(2M+3)/2} \left(1- \left(1-\frac{1}{k} \right)^2 \right)^{-M} \\
&\leq& C' k^{-(2M+3)/2} \left(1- \left(1-\frac{1}{k} \right) \right)^{-M} \\
&=& C' k^{-3/2}.
\end{eqnarray*}
And thus
\begin{eqnarray*}
E &\leq& \sum_{k=2}^\infty E_k + \int \int_U \frac{|f(\xi)| |\hat{f}(\eta)|}{(1+|\xi|+|\eta|)^{2M+3}}
e^{|\xi \cdot \eta|} \, d\xi \, d\eta \\
&\leq& C' \sum_{k=2}^\infty k^{-3/2} + e \|f\|_1 \|\hat{f}\|_1 < \infty.
\end{eqnarray*}
\end{proof}

As mentioned in the introduction, we now prove the following result which clearly shows that an exact analogue of Vemuri's theorem is not true in the context of Beurling's theorem. In what follows we denote by  $\Phi_\alpha$ the normalised Hermite functions on $ \R^n.$

\begin{thm} \label{Beurling-Vemuri}
Let $ 0 < a < 1 $ be fixed. There does not exist any $ t > 0 $ such that the Hermite coefficients of every  function $ f $ with the subcritical Beurling's condition
\begin{eqnarray*}
K_a(f) = \int_{\R^n}\int_{\R^n} |f(x)| |\hat{f}(y)| e^{a|x \cdot y|} \, dx \, dy < \infty \end{eqnarray*}
satisfy $ |(f, \Phi_\alpha)| \leq C_t e^{-(2|\alpha|+n)t/2} $ for all $ \alpha \in \N^n .$
\end{thm}
\begin{proof}
Suppose, to the contrary, that there exists $ t > 0 $ such that $ |(f, \Phi_\alpha)| \leq C_{f,t} e^{-(2|\alpha|+n)t/2}, $ for all $ \alpha \in \N^n$ and for all functions $ f $ satisfying  $K_a(f) < \infty.$
For each $\delta > 0,$ let $f_{\delta}$ stand for the dilation of $f$ given by $f_{\delta}(x) = \delta^{n/2} f(\delta x).$ Then $(f_{\delta})^{\wedge}(\xi) = {\delta}^{-n/2} \hat{f}({\delta}^{-1} \xi)$ and $K_a(f_{\delta}) = K_a(f) < \infty.$ Thus for all $\delta >0$, we should have
\begin{equation*}
|(f_{\delta}, \Phi_\alpha)| \leq C_{f, \delta, t}~ e^{-(2|\alpha|+n)t/2},~ \textup{ for all }\alpha \in \N^n.
\end{equation*}
Now we use the above estimate in the Hermite expansion of $f_{\delta}$ to get
\begin{equation*}
|f_{\delta}(x)| \leq  C_{f, \delta, t} \sum_{\alpha \in \N^n} e^{-(2|\alpha|+n)t/2} |\Phi_{\alpha}(x)|.
\end{equation*}
If we apply Holder's inequality in the above estimate and use the Mehler's formula for Hermite functions (see Proposition 1.2.1 in \cite{Th2}), viz., for all $r<1$,
\begin{equation*}
\sum_{\alpha \in \N^n} r^{|\alpha|} \Phi_{\alpha}(x) \Phi_{\alpha}(y) = \pi^{-n/2} (1-r^2)^{-n/2} e^{-\frac{1}{2} \frac{1+r^2}{1-r^2} (|x|^2 + |y|^2) + \frac{2r}{1-r^2}x \cdot y},
\end{equation*}
then we get that for every $s < t$,
\begin{equation*}
|f_{\delta}(x)| \leq C_1 e^{-\frac{1}{2} \tanh (s)|x|^2}
\end{equation*}
for some constant $C_1$ which is independent of $x$. Fix an $s < t$ and choose $\delta_0$ such that $ \tanh(s) > \delta_0.$ 
Then we have
\begin{equation} \label{Beurling-Vemuri-eq1}
|f(x)| = \delta_0^{-n/2} |f_{\delta_0}({\delta_0}^{-1}x)| \leq C_1 \delta_0^{-n/2} e^{-\frac{1}{2} \delta^{-2}_0 \tanh (s)|x|^2}.
\end{equation}
Since $|(f_\delta, \Phi_\alpha)| = |((f_\delta)^{\wedge}, \Phi_\alpha)|,$ similar estimate holds for $(f_\delta)^{\wedge}$ as well. In particular for $\delta = 1$,
\begin{equation} \label{Beurling-Vemuri-eq2}
|\hat{f}(\xi)| \leq C_2 e^{-\frac{1}{2} \tanh (s)|\xi|^2}.
\end{equation}
With (\ref{Beurling-Vemuri-eq1}) and (\ref{Beurling-Vemuri-eq2}), Hardy's theorem can be applied to conclude that $f = 0.$
\end{proof}

We now prove the following results for eigenfunctions of the Fourier transform.

\begin{thm}
Let $0<a<1$ be fixed. Let $f \in L^2(\R^n)~ (n \geq 1)$ satisfy the subcritical Beurling's condition
$$ K_a(f) = \int_{\R^n}\int_{\R^n} |f(x)| |\hat{f}(y)| e^{a|x \cdot y|} \, dx \, dy < \infty.$$
If, in addition, $f$ is an eigenfunction of the Fourier transform, then there exists a positive constant $C$ (independent of $a$) such that for all $\alpha \in \N^n$,
$$ \left| (f,\Phi_\alpha) \right| \leq C e^{t/2} (K_a(f))^{1/2} \left( \prod_{j=1}^n (2\alpha_j +1)^{1/4n} \right) e^{-(2|\alpha|+n)t/2n}$$
where $t$ is determined by the condition $a = \tanh(2t)$.
\end{thm}

\begin{proof}
Once again we make use of the Bargmann transform $ B.$ The most important property of $B$ which we need is that the Taylor coefficients $c_\alpha$ of $Bf$ are related to the Hermite coefficients $(f,\Phi_\alpha)$ of $f.$ More precisely, we have
$$ (f,\Phi_\alpha) = \left( 2^\alpha \alpha!~\pi^{n/2} \right)^{1/2} c_\alpha.$$
Therefore, in order to prove the theorem we only need to estimate the Taylor coefficients of $ Bf $ for which, in view of Cauchy's formula, we need good estimates of $ Bf.$ We have already seen in the proof of Theorem \ref{*} that the given assumption on $f$ leads to
$$ |Bf(z)B\hat{f}(z)| \leq \pi^{-n} K_a(f) \exp \left( \frac{1}{2} \left(|y|^2 + \frac{1-a}{1+a}|x|^2 \right) \right).$$
By assumption $f$ is an eigenfunction of the Fourier transform, therefore $|B\hat{f}(z)| = |Bf(z)|.$ Thus we have
$$ |Bf(z)| \leq \left( \pi^{-n} K_a(f) \right)^{1/2} \exp \left( \frac{1}{4} \left(|y|^2 + \frac{1-a}{1+a}|x|^2 \right) \right).$$
But since $B\hat{f}(z) = Bf(-iz)$ and $|B\hat{f}(z)| = |Bf(z)|$ we also have
$$ |Bf(z)| \leq \left( \pi^{-n} K_a(f) \right)^{1/2} \exp \left( \frac{1}{4} \left(|x|^2 + \frac{1-a}{1+a}|y|^2 \right) \right).$$
In the one dimensional case, we can apply Phragm$\acute{\textup{e}}$n-Lindel\"{o}f principle (see the proof of Theorem 2.1 in \cite{V}) to prove that
$$ |Bf(z)| \leq \left( \pi^{-1} K_a(f) \right)^{1/2} \exp \left( \frac{1}{4} \sqrt{\frac{1-a}{1+a}} \left(|x|^2 + |y|^2 \right) \right).$$
And proceeding with the proof of Theorem 2.1 in \cite{V} we get that
\begin{eqnarray*}
\left| \left( f,h_k \right) \right| \leq C' e^{t/2} \left( K_a(f) \right)^{1/2} \left( 2k+1 \right)^{1/4} e^{-(2k+1)t/2}
\end{eqnarray*}
where $C'$ is independent of $a$ and $t$ is determined by the condition $a = \tanh(2t)$.

In higher dimensions ($n \geq 2$), for each fixed $z_2, \ldots, z_n \in \C$, we think of $Bf(\cdot, z_2, \ldots, z_n)$ as an entire function of one complex variable which is bounded by
\begin{eqnarray*}
\left( \pi^{-n} K_a(f) \right)^{1/2} \prod_{j=2}^n \exp \left( \frac{1}{4} |z_j|^2 \right) \exp \left( \frac{1}{4} \left(y_1^2 + \frac{1-a}{1+a}x_1^2 \right) \right)
\end{eqnarray*}
and
\begin{eqnarray*}
\left( \pi^{-n} K_a(f) \right)^{1/2} \prod_{j=2}^n \exp \left( \frac{1}{4} |z_j|^2 \right) \exp \left( \frac{1}{4} \left(x_1^2 + \frac{1-a}{1+a}y_1^2 \right) \right).
\end{eqnarray*}
Then we can proceed as in \cite{V} (as mentioned above) to get
\begin{eqnarray*}
\left| \left( f,\Phi_\alpha \right) \right| \leq C'' e^{t/2} \left( K_a(f) \right)^{1/2} \left( 2\alpha_1+1 \right)^{1/4} e^{-(2\alpha_1+1)t/2}
\end{eqnarray*}
where $C''$ is independent of $a$. Similarly we can consider other variables of $Bf$ as well. Thus for each $j \in \{1, 2, \ldots, n \}$ we get
\begin{eqnarray*}
\left| \left( f,\Phi_\alpha \right) \right| \leq C'' e^{t/2} \left( K_a(f) \right)^{1/2} \left( 2\alpha_j+1 \right)^{1/4} e^{-(2\alpha_j+1)t/2}.
\end{eqnarray*}
By combining these estimates together, we get that
\begin{eqnarray*}
\left| (f,\Phi_\alpha) \right| \leq C'' e^{t/2} (K_a(f))^{1/2} \left( \prod_{j=1}^n (2\alpha_j +1)^{1/4n} \right) e^{-(2|\alpha|+n)t/2n}.
\end{eqnarray*}
\end{proof}

\begin{rem}
As mentioned earlier, the conclusion of the above theorem in the higher dimensional case is believed not to be the best, the reason being (upto our knowledge) the absence of an appropriate analogue of Phragm$\acute{\textup{e}}$n-Lindel\"{o}f principle. However in a special case we do obtain the best decay.
\end{rem}

\begin{thm}
Let $0<a<1$ be fixed. Let $f \in L^2(\R^n)~ (n \geq 2)$ satisfy the subcritical Beurling's condition
\begin{eqnarray*}K_a(f) = \int_{\R^n}\int_{\R^n} |f(x)| |\hat{f}(y)| e^{a|x \cdot y|} \, dx \, dy < \infty.
\end{eqnarray*}
If, in addition, $f$ is $O(n)-$finite and an eigenfunction of the Fourier transform, then there exists a positive constant $C$ (independent of $a$) such that for all $\alpha \in \N^n$,
\begin{eqnarray*}
\left| (f,\Phi_\alpha) \right| \leq C (K_a(f))^{1/2} \left( \prod_{j=1}^n (2\alpha_j +1)^{1/4} \right) e^{-(2|\alpha|+n)t/2}
\end{eqnarray*}
where $t$ is determined by the condition $a = \tanh(2t).$
\end{thm}
\begin{proof}
We have already seen in the previous theorem that the given assumptions on $f$ lead to the following estimates on $Bf$
\begin{eqnarray*}
|Bf(z)| &\leq& \left( \pi^{-n} K_a(f) \right)^{1/2} \exp \left( \frac{1}{4} \left(|y|^2 + \frac{1-a}{1+a}|x|^2 \right) \right), \\
|Bf(z)| &\leq& \left( \pi^{-n} K_a(f) \right)^{1/2} \exp \left( \frac{1}{4} \left(|x|^2 + \frac{1-a}{1+a}|y|^2 \right) \right).
\end{eqnarray*}
With these estimates, by considering a vector valued Bargmann transform, we have proved the result in Section 5 of \cite{RT1}.
\end{proof}

Finally we prove that any function $f$ satisfying Beurling's condition is an entire vector for the Schr\"odinger representation of the Heisenberg group $\H^n$. The space of all analytic vectors for the Schr\"odinger representation $\pi$ of the Heisenberg group is denoted by $(L^2(\R^n))^\omega$. Recall that $f \in L^2(\R^n)$ is said to be an analytic vector for $\pi$ if the map $(x,u) \rightarrow \pi(x,u)f $ is real analytic from $\R^{2n} $ into $ L^2(\R^n)$. It is known that $f \in (L^2(\R^n))^\omega$ if and only if the map $(x,u) \rightarrow \pi(x,u)f $ extends as a holomorphic function to a tube domain in $ \C^{2n} $ such that
\begin{equation} \label{entire-vector-eq1}
\sup_{\{(y,v) : |y|^2 + |v|^2 < t^2\}} \int_{\R^n} |\pi(x+iy,u+iv)f(\xi)|^2 \, d\xi < \infty
\end{equation}
for some $ t> 0.$ If the above is true, then we have the formula (see Theorem 1.1 in \cite{Th1})
$$ \int_{\R^n} \int_{K}|\pi(k \cdot (x+iy,u+iv)) f(\xi)|^2 \, dk \, d\xi $$
$$ = e^{(u\cdot y-v\cdot x)}\sum_{k=0}^\infty \|P_kf\|_2^2 \frac{k!(n-1)!} {(k+n-1)!} \varphi_k^{n-1}(2iy,2iv)$$
where $K$ is the compact group $Sp(n,\R) \cap O(2n,\R), $ and $ P_k $ are the orthogonal projections associated to the eigenspaces of the Hermite operator $ H $ and $ \varphi_k^{n-1}(y,v) = L_k^{n-1}(\frac{1}{2}(|y|^2+|v|^2)) e^{-\frac{1}{4}(|y|^2+|v|^2)} $ are the Laguerre functions of type $ (n-1).$ Here $Sp(n,\R)$ stands for the symplectic group of $2n$ by $2n$ matrices with real entries that preserve the symplectic form $[(x,u),(y,v)] = (u \cdot y - v \cdot x)$ on $R^{2n}$ and have determinant one. The asymptotic behavior of the Laguerre functions $ \varphi_k^{n-1} $ in the complex domain is well known. From Theorem 8.22.3 in Szego \cite{Sz} we see that $\varphi_k^{n-1} (2iy,2iv) $ behaves like $e^{2\sqrt{(2k+n)}(|y|^2+|v|^2)^{1/2}}.$ From the above formula it follows that $ f \in (L^2(\R^n))^\omega $ if and only if $f$ is in the image of $L^2(\R^n)$ under the Hermite-Poisson semigroup $e^{-t\sqrt{H}}$ for some $t > 0.$ A function $f \in L^2(\R^n)$ is said to be an entire vector for $\pi$ if the map $(x,u) \rightarrow \pi(x,u)f $ extends to $\C^{2n}$ as an entire function and satisfies (\ref{entire-vector-eq1}) for every $t>0$. It follows then from the above that $f$ is an entire vector for $\pi$ if and only if $f$ is in the image of $L^2(\R^n)$ under $e^{-t\sqrt{H}}$ for every $t > 0.$ With this we are now ready to prove the following theorem:

\begin{thm}
Let $f \in L^2(\R^n)$ satisfy the subcritical Beurling's condition
\begin{eqnarray*}
K_a(f) = \int_{\R^n}\int_{\R^n} |f(x)| |\hat{f}(y)| e^{a |x \cdot y|} \, dx \, dy < \infty
\end{eqnarray*}
for some $  a > 0.$ Then $f$ is an entire vector for the Schr\"odinger representation of the Heisenberg group $\H^n$.
\end{thm}
\begin{proof}
We have noticed that $f$ is an entire vector for $\pi$ if and only if it is in the image of $L^2(\R^n)$ under $e^{-t\sqrt{H}}$ for every $t > 0.$ Therefore, the theorem will follow once we prove that
\begin{eqnarray*}
|(f, \Phi_\alpha)| \leq C_{t} e^{-t(2|\alpha| + n)^{1/2}}
\end{eqnarray*}
for all $\alpha \in \N^n$ and for all $t > 0.$ For $a \geq 1,$ we already know that $f = 0.$ So we take $0<a<1.$ In this case there are infinitely many linearly independent functions satisfying the Beurling's condition, e.g. all Hermite functions. Without loss of generality let us assume $ f $ is nontrivial. The idea of the proof of this theorem comes from the proof of Beurling's theorem due to H\"{o}rmander \cite{H}. Consider the non-negative functions $A$ and $B$ on $\R^n$ defined by
\begin{eqnarray*}
A(\xi) = \int_{\R^n} |\hat{f}(\eta)| e^{a |\xi \cdot \eta|} \, d\eta,~ B(\eta) = \int_{\R^n} |f(\xi)| e^{a |\xi \cdot \eta|} \, d\xi.
\end{eqnarray*}
Then by the hypothesis of the theorem
\begin{eqnarray*}
\int_{\R^n} |f(\xi)| A(\xi) \, d\xi = \int_{\R^n} |\hat{f}(\eta)| B(\eta) \, d\eta = K_a(f) < \infty.
\end{eqnarray*}
Now repeating the arguments of the first part of Lemma \ref{***}, one can verify that there exists some $\delta' > 0$ such that
$$ \int_{\R^n} |f(\xi)| e^{\delta'|\xi|} \, d\xi < \infty,~
\int_{\R^n} |\hat{f}(\eta)| e^{\delta'|\eta|} \, d\eta < \infty.$$
As a result of this, both $f$ and $\hat{f}$ extend as holomorphic functions to a domain $\Omega^n$ in $\C^n$ containing $\R^n,$ where $\Omega$ is a horizontal strip in $\C$ containing $\R.$

Now we will show that both $A$ and $B$ grow faster than any exponential function. For this first notice that $S^{n-1}$ can be written as a union of finitely many proper open connected spherical caps $U_j,~ j=1,2,...,N$ such that for each $j$,
\begin{eqnarray*}
\left| \omega_1 \cdot \omega_2 \right| \geq \frac{1}{2},~~~~~~ \textup{ for all } \omega_1, \omega_2 \in U_j.
\end{eqnarray*}
Now choose $k$ for which $\xi' \in U_k$. Then
\begin{eqnarray*}
A(r\xi') &\geq& \int_0^\infty \int_{U_k} |\hat{f}(s \omega)| e^{\frac{a}{2} rs} s^{n-1} \, ds \, d\omega.
\end{eqnarray*}
Since $S^{n-1}$ is the union of $U_j$'s, it follows that
\begin{eqnarray*}
A(r\xi') &\geq& \min_j \int_0^\infty \int_{U_j} |\hat{f}(s \omega)| e^{\frac{a}{2} rs} s^{n-1} \, ds \, d\omega \\
&\geq& \min_j \int_{s \geq \frac{2t}{a}} \int_{U_j} |\hat{f}(s \omega)| e^{\frac{a}{2} rs} s^{n-1} \, ds \, d\omega \\
&\geq& e^{tr} \min_j \int_{s \geq \frac{2t}{a}} \int_{U_j} |\hat{f}(s \omega)| s^{n-1} \, ds \, d\omega.
\end{eqnarray*}
The above is true for every $t>0$.
Since $\hat{f}$ is the restriction to $\R^n$ of a non-zero complex analytic function in a domain $\Omega^n \subset \C^n$ containing $\R^n$, it follows that $\hat{f}$ is a real analytic function on $\mathbb{R}^n$. Thus by uniqueness theorem of real analytic functions, it follows immediately that in any given domain of $\R^n,$ one can find a smaller domain on which $\hat{f}$ is away from zero. Thus for any $j \in \{1,2,....,n\},$
\begin{eqnarray*}
\int_{s \geq \frac{2t}{a}} \int_{U_j} |\hat{f}(s \omega)| s^{n-1} \, ds \, d\omega >0.
\end{eqnarray*}  
But then
\begin{eqnarray*}
A(r\xi') \geq C(t) e^{tr}
\end{eqnarray*}
where
\begin{eqnarray*}
C(t) = \min_j \int_{s \geq \frac{2t}{a}} \int_{U_j} |\hat{f}(s \omega)| s^{n-1} \, ds \, d\omega > 0.
\end{eqnarray*}

Similarly $B$ grows faster than any exponential function. In other words for every $t>0$
\begin{eqnarray*}
\int_{\R^n} |f(\xi)| e^{t|\xi|} \, d\xi < \infty,~~~ \int_{\R^n} |\hat{f}(\eta)| e^{t|\eta|} \, d\eta < \infty.
\end{eqnarray*}
%
%
With these estimates the proof will be completed once we have the following theorem.
\end{proof}

\begin{thm}
Let $\psi$ be a measurable function on $\R^n$ satisfying the estimates
\begin{eqnarray*}
|\psi(\xi)| \leq g(\xi) e^{-t|\xi|},~~~ |\hat{\psi}(\eta)| \leq h(\eta) e^{-t|\eta|}
\end{eqnarray*}
for some $t > 0$ and integrable functions $g$ and $h.$ Then there exists a positive constant $C$ (independent of $a$) such that for all $\alpha \in \N^n$,
\begin{eqnarray*}
|(\psi, \Phi_\alpha)| \leq C \prod_{j=1}^n (2\alpha_j+1)^{1/4}~ e^{-\frac{t}{\sqrt{2n}}(2\alpha_j+1)^{1/2}}
\end{eqnarray*}
where $t$ is determined by the condition $a = \tanh(2t).$
\end{thm}
\begin{proof}
The theorem is proved by estimating the Taylor coefficients of the Bargmann transform of $\psi.$ An exact analogue of this theorem was proved in \cite{RT2} (Theorem 3.9) where $g$ and $h$ were assumed to be bounded. Nevertheless, we present the proof here for the convenience of the reader.
The given condition on $\psi$ implies that
\begin{eqnarray*}
|B\psi(x+iy)| \leq \pi^{-n/2} e^{-\frac{1}{4}(|x|^2 - |y|^2)} \int_{\R^n} g(\xi) e^{-t|\xi|} e^{- \frac{1}{2}|\xi|^2} e^{|x||\xi|} \, d\xi.
\end{eqnarray*}
By completing the squares, one can verify that the above integral is dominated by $ e^{\frac{1}{2}(|x|-t)^2} \| g\|_1$ leading to
\begin{eqnarray*}
|B\psi(x+iy)| &\leq& C_{n,t}~ e^{\frac{1}{4}(|x|^2+|y|^2)} e^{-t|x|} \\
&\leq& C_{n,t}~ \prod_{j=1}^n e^{\frac{1}{4}(x_j^2+y_j^2)} e^{-\frac{t}{\sqrt{n}}|x_j|}.
\end{eqnarray*}
Similarly, the given condition on $\hat{\psi}$ and the relation $B\psi(-iz) = B\hat{\psi}(z)$ gives the other estimate for $B\psi$, namely,
$$ |B\psi(x+iy)| \leq C_{n,t}~ \prod_{j=1}^n e^{\frac{1}{4}(x_j^2+y_j^2)} e^{-\frac{t}{\sqrt{n}}|y_j|}.$$
Using the Cauchy integral formula, we get for every $r_j > 0,~ j=1, 2, \ldots, n$
\begin{eqnarray*}
|c_\alpha| &\leq& \frac{1}{(2\pi)^n} \int_0^{2\pi} \ldots \int_0^{2\pi} \frac {|B\psi(r_1 e^{i\theta_1}, \ldots, r_n e^{i\theta_n})|} {r_1^{\alpha_1} \ldots r_n^{\alpha_n}} \, d\theta_1 \ldots \, d\theta_n \\
&\leq& \frac{4 C_{n,t}}{(2\pi)^n} \prod_{j=1}^{n} r_j^{-\alpha_j} e^{\frac{1}{4}r_j^2} \left( \int_0^{\frac{\pi}{4}} e^{-\frac{t}{\sqrt{n}} r_j \cos \theta_j} \, d\theta_j + \int_{\frac{\pi}{4}}^{\frac{\pi}{2}} e^{-\frac{t}{\sqrt{n}} r_j \sin \theta_j} \, d\theta_j \right) \\
&\leq& \tilde{C}_{n,t} \prod_{j=1}^{n} r_j^{-\alpha_j} e^{\frac{1}{4}r_j^2} e^{-\frac{t}{\sqrt{2n}}r_j}.
\end{eqnarray*}
Since the above is true for every $r_j > 0,~ j=1, 2, \ldots, n$, we can take in particular $r_j = (2\alpha_j+1)^{1/2}$ to get
\begin{eqnarray*}
|c_\alpha| &\leq& \tilde{C}_{n,t} \prod_{j=1}^n (2\alpha_j+1)^{-\alpha_j/2}~
e^{\frac{1}{4}(2\alpha_j+1)} e^{-\frac{t}{\sqrt{2n}} (2\alpha_j+1)^{1/2}}.
\end{eqnarray*}
Thus
\begin{eqnarray*}
|(\psi,\Phi_\alpha)| &\leq& \tilde{C}_{n,t} \left( 2^\alpha \alpha! \pi^{n/2} \right)^{1/2} \prod_{j=1}^n (2\alpha_j+1)^{-\alpha_j/2}~ e^{\frac{1}{4}(2\alpha_j+1)} e^{-\frac{t}{\sqrt{2n}}(2\alpha_j+1)^{1/2}} \\
&\sim& \tilde{C}_{n,t} \prod_{j=1}^n (2\alpha_j+1)^{1/4}~ e^{-\frac{t}{\sqrt{2n}}(2\alpha_j+1)^{1/2}}
\end{eqnarray*}
where the last estimate is obtained using the Stirling's formula $\Gamma(\lambda+1) \sim \lambda^{\lambda+1/2}~ e^{-\lambda}.$
\end{proof}
  
\begin{center}
{\bf Acknowledgments}

\end{center}
This work is supported in part by grant from UGC Centre for advanced Study. The work of first author is supported in part by SRF from CSIR, India. The work of second author is supported in part by J.C. Bose National Fellowship from DST, India.

\end{document}